\documentclass[reqno]{amsart}

\usepackage{amssymb}
\usepackage{amstext}
\usepackage{amsmath}
\usepackage{amscd}
\usepackage{amsthm}
\usepackage{amsfonts}
\usepackage{enumerate}
\usepackage{graphicx}
\usepackage{latexsym}
\usepackage{mathrsfs}
\usepackage[all]{xy}
\usepackage{mathtools}
\usepackage{upgreek}
\usepackage{enumitem}
\usepackage{dsfont}
\usepackage{colonequals}
\usepackage{tikz}

\numberwithin{equation}{section}

\usepackage{setspace}

\setlist[enumerate]{format=\normalfont}

\newcommand{\marginparstretch}{0.6}
\let\oldmarginpar\marginpar
\renewcommand\marginpar[1]{\-\oldmarginpar[\framebox{\setstretch{\marginparstretch}\begin{minipage}{\marginparwidth}{\raggedleft\tiny #1}\end{minipage}}]{\framebox{\setstretch{\marginparstretch}\begin{minipage}{\marginparwidth}{\raggedright\tiny #1}\end{minipage}}}}

\usepackage[colorlinks]{hyperref}

\newtheorem{theorem}{Theorem}[section]

\newtheorem{corollary}[theorem]{Corollary}
\newtheorem{lemma}[theorem]{Lemma}
\newtheorem{proposition}[theorem]{Proposition}
\newtheorem{definition-proposition}[theorem]{Definition-Proposition}

\theoremstyle{definition}
\newtheorem{definition}[theorem]{Definition}

\newtheorem{remark}[theorem]{Remark}
\newtheorem{example}[theorem]{Example}

\newcommand{\mm}{{\mathfrak{m}}}
\newcommand{\nn}{{\mathfrak{n}}}
\newcommand{\pp}{{\mathfrak{p}}}

\newcommand{\CC}{\mathcal{C}}

\newcommand{\DD}{\mathcal{D}}

\newcommand{\ZZ}{\mathcal{Z}}

\newcommand{\Z}{\mathbb{Z}}
\newcommand{\Q}{\mathds{Q}}

\newcommand{\depth}{\operatorname{depth}\nolimits}

\newcommand{\Ext}{\operatorname{Ext}\nolimits}
\newcommand{\Tor}{\operatorname{Tor}\nolimits}
\newcommand{\Hom}{\operatorname{Hom}\nolimits}
\newcommand{\End}{\operatorname{End}\nolimits}

\newcommand{\gl}{\mathop{{\rm gl.dim}}\nolimits}
\newcommand{\injdim}{\mathop{{\rm inj.dim}}\nolimits}
\newcommand{\op}{\operatorname{op}\nolimits}
\newcommand{\RHom}{\mathbf{R}\strut\kern-.2em\operatorname{Hom}\nolimits}

\newcommand{\Spec}{\operatorname{Spec}\nolimits}
\newcommand{\Supp}{\operatorname{Supp}\nolimits}

\newcommand{\Zn}{\mathbb{Z}_n}

\newcommand{\order}[1]{ \left| #1 \right| }

\DeclareMathOperator{\moduleCategory}{\mathsf{mod}} \renewcommand{\mod}{\moduleCategory}
\DeclareMathOperator{\Mod}{\mathsf{Mod}}
\DeclareMathOperator{\proj}{\mathsf{proj}}

\DeclareMathOperator{\CM}{\mathsf{CM}}
\DeclareMathOperator{\Cl}{\mathsf{Cl}}

\renewcommand{\det}{\operatorname{\mathsf{det}}\nolimits}
\DeclareMathOperator{\refl}{\mathsf{ref}}
\DeclareMathOperator{\add}{\mathsf{add}}

\newcommand\mypound{\protect\scalebox{0.7}{\protect\raisebox{0.4ex}{\rm\#}}}
\newcommand{\NCov}[1]{#1\mypound G}
\newcommand\smypound{\protect\scalebox{0.5}{\protect\raisebox{0.4ex}{\rm\#}}}
\newcommand{\sNCov}[1]{#1\smypound G}

\DeclareMathOperator{\rk}{rank}
\DeclareMathOperator{\len}{length}
\DeclareMathOperator{\hei}{height}
\DeclareMathOperator{\NCCR}{\mathsf{NCCR}}
\DeclareMathOperator{\GM}{\mathsf{GM}}

\begin{document}
\title[Gorenstein modifications and $\Q$-Gorenstein rings]{Gorenstein modifications and $\Q$-Gorenstein rings}
\author{Hailong Dao, Osamu Iyama, Ryo Takahashi and Michael Wemyss}
\address{H. Dao: Department of Mathematics, University of Kansas, Lawrence, KS 66045-7523, USA}
\email{hdao@ku.edu}
\urladdr{http://www.math.ku.edu/~hdao/}

\address{O. Iyama: Graduate School of Mathematics, Nagoya University, Furocho, Chikusaku, Nagoya 464-8602, Japan}
\email{iyama@math.nagoya-u.ac.jp}
\urladdr{http://www.math.nagoya-u.ac.jp/~iyama/}

\address{R. Takahashi: Graduate School of Mathematics, Nagoya University, Furocho, Chikusaku, Nagoya 464-8602, Japan}
\email{takahashi@math.nagoya-u.ac.jp}
\urladdr{http://www.math.nagoya-u.ac.jp/~takahashi/}

\address{Michael Wemyss: School of Mathematics and Statistics, University of Glasgow, 15 University Gardens, Glasgow, G12 8QW,
UK.}
\email{michael.wemyss@glasgow.ac.uk}
\urladdr{http://www.maths.gla.ac.uk/~mwemyss/}

\begin{abstract}
Let $R$ be a Cohen--Macaulay normal domain with a canonical module $\omega_R$.  It is proved that if $R$ admits a noncommutative crepant resolution (NCCR), then necessarily it is $\mathds{Q}$-Gorenstein.  Writing $S$ for a Zariski local canonical cover of $R$, then a tight relationship between the existence of noncommutative (crepant) resolutions on $R$ and $S$ is given.  A weaker notion of Gorenstein modification is developed, and a similar tight relationship is given.  There are three applications: non-Gorenstein quotient singularities by connected reductive groups cannot admit an NCCR,  the centre of any NCCR is log-terminal, and the Auslander--Esnault classification of two-dimensional CM-finite algebras can be deduced from Buchweitz--Greuel--Schreyer.
\end{abstract}
\maketitle


\section{Introduction}

The theory of cyclic covers of varieties is now well established in algebraic geometry, particularly in birational geometry and the minimal model programme \cite{c3f,K,KM}.  Under mild assumptions, cyclic covers reduce problems involving Cohen--Macaulay rings to problems involving Gorenstein rings, which are typically much easier.  As such, a cyclic cover of a variety often improves its singularities, and it is a classical problem both in birational geometry and in commutative algebra to study this relationship.

On the other hand, noncommutative approaches to problems in birational geometry have recently been developed \cite{VdB1d, VdBNCCR, IW1, HomMMP} through noncommutative resolutions and their variants.  This paper considers how these structures behave under taking cyclic covers, and we give characterisations of how each of the various noncommutative structures can be detected and constructed on the cover.  Along the way, we give a general algebraic obstruction to the existence of an NCCR, through the $\mathds{Q}$-Gorenstein property.

\subsection{Summary of Results}
Throughout this introduction, let $R$ be a commutative noetherian ring. It will not be assumed that $R$ is local, normal or Cohen--Macaulay, unless stated otherwise. We write  $\CM R$ for the category of Cohen--Macaulay (CM) $R$-modules, that is, finitely generated $R$-modules $X$ satisfying $\depth_{R_{\pp}}(X\otimes_RR_{\pp})=\dim R_{\pp}$ for all $\pp\in\Supp_RX$, and write $\refl R$ for the category of finitely generated reflexive $R$-modules. 

For $M\in\refl R$, we say that $M$ gives a \emph{noncommutative crepant resolution} (NCCR) \emph{of $R$} if $A\colonequals \End_R(M)$ satisfies $A\in\CM R$ and $\gl (A\otimes_RR_{\pp})=\dim R_{\pp}$ for all $\pp\in\Supp_RA$.  These two conditions are independent, in contrast to commutative rings, and we write $\NCCR R$ for the set of additive equivalence classes of $R$-modules giving NCCRs of $R$, where  $R$-modules $X$ and $Y$ are \emph{additively equivalent} if $\add X=\add Y$.

When $R$ is a Cohen--Macaulay normal domain with a fixed canonical module $\omega_R$, recall that $R$ is \emph{$\Q$-Gorenstein of index $n>0$} if $\omega_R^n$ is a locally free $R$-module, where $\omega_R^i\colonequals (\omega_R\otimes\cdots\otimes\omega_R)^{**}$ ($i$ times), and $n$ is the smallest number with this property.  We will also consider the stronger property that $R$ is \emph{$\mathds{Q}$-Calabi--Yau} ($\Q$-CY) \emph{of index $n>0$}, namely when $\omega_R^n\cong R$ and $n$ is the smallest number with this property. This is equivalent to $\mathds{Q}$-Gorenstein of index $n$ when $R$ is local. In contrast to $\Q$-Gorenstein, $\Q$-CY depends on a choice of $\omega_R$.

Setting notation for our first result, let $\Zn:=\Z/n\Z$ and $S$ be a $\Zn$-graded ring.  Then an $S$-module $X$ is \emph{$\Zn$-gradable} if there exists a $\Zn$-graded $S$-module $Y$ which is isomorphic to $X$, as an $S$-module. We write $\NCCR_{\Zn}\!S$ for the subset of $\NCCR\!S$ consisting of additive equivalence classes of $\Zn$-gradable $S$-modules.  Our first main result is then the following, where we refer the reader to Theorems~\ref{prop 1} and \ref{main} for the detailed statements.

\begin{theorem}[Theorems~\ref{prop 1} and \ref{main}]\label{main NCCR}
Let $R$ be a CM normal domain, admitting  a canonical module, and containing a field of characteristic zero.
\begin{enumerate}
\item If $R$ has an NCCR, then the following conditions hold.
\begin{enumerate}
\item $R$ is $\Q$-Gorenstein.
\item Every Zariski local canonical cover of $R$ is Gorenstein and has an NCCR given by a gradable module.
\end{enumerate}
\item Conversely, if the following conditions are satisfied, then $R$ has an NCCR.
\begin{enumerate}
\item $R$ is $\Q$-CY.
\item A canonical cover $S$ of $R$ is Gorenstein and has an NCCR given by a gradable module.
\end{enumerate}
In this case, for the index $n$ of the cover, there is a bijection 
\[
\NCCR R\simeq\NCCR_{\Zn}\!S.
\]
\end{enumerate}
\end{theorem}

A main ingredient used in the proof, which controls the gap between a $\Q$-Gorenstein ring and its canonical cover, is given by a certain noncommutative ring which we call a \emph{Gabriel cover} to avoid any confusion.
This well known concept in representation theory is usually called as a Galois covering \cite{G} or a smash product \cite{CM}; see also \cite{As,DGL}.  In Sections~\ref{Gabriel} and \ref{Gabriel2}, we include a brief introduction to Gabriel covers and explain why they are a standard tool to deal with graded modules.

The above theorem settles the NCCR situation, but there are other important noncommutative variants of interest.  In this paper our primary interest lies in what we call \emph{Gorenstein modifications} (GM).  For $M\in\refl R$ with $A:=\End_R(M)$, we say that $A$ is a Gorenstein modification of $R$ if $A\in\CM R$ and $\injdim_{A\otimes_RR_\pp}(A\otimes_RR_{\pp})=\dim R_{\pp}$ for all $\pp\in\Supp_RA$. We remark that the inequality $\ge$ always holds \cite[Corollary 3.5(4)]{GN2}. Also, when $R$ is CM and admits a canonical module, GMs appear in the literature in many different guises; see \S\ref{Sect 3.2} for an overview. However the setting in this paper is often more general, and at times requires our more general definition.  Recall also that for $M\in\refl R$ with $A:=\End_R(M)$, we say that $M$ gives an \emph{noncommutative resolution} (NCR) \emph{of $R$} if $\gl (A\otimes_RR_{\pp})<\infty$ for all $\pp\in\Spec R$ (cf. \cite{DITV}).

Now assume that $R$ is a normal domain with a canonical module $\omega_R$. Then the category $\refl R$ admits two relevant dualities $(-)^*\colonequals \Hom_R(-,R)\colon\refl R\to\refl R$ and $(-)^{\vee}\colonequals \Hom_R(-,\omega_R)\colon\refl R\to\refl R$. A key role in this paper is played by the equivalence
\[
\upnu\colonequals (-)^{\vee}\circ(-)^*=(\omega_R\otimes_R-)^{**}\colon\refl R\to\refl R
\]
which is called the \emph{Nakayama functor}. It gives the Auslander-Reiten translation in dimension two \cite{A2,Y}, and the higher Auslander-Reiten translation in any dimension \cite{I}. The first part of the following result is an analog of results in \cite{AS,IS}, and the second part generalizes \cite[Theorem 2.3.2]{I}.

\begin{proposition}[Proposition \ref{prop 0 text}]\label{prop 0}
Let $R$ be a CM normal domain with a canonical module.
Assume that $M\in\refl R$ gives a modification of $R$. Then $M$ gives a GM of $R$ if and only if $\add M=\add \upnu(M)$ holds. In this case, $\upnu$ gives an autoequivalence on $\add M$.
\end{proposition}

We write $\GM R$ for the set of additive equivalence classes of $R$-modules giving GMs of $R$.  When $S$ is a $\Zn$-graded ring, as before we define the subset $\GM_{\Zn}\!S$ of $\GM\!S$ in the obvious way.  Our next  main result is the following, where again we refer the reader to Theorems~\ref{prop 1} and \ref{main} for the detailed statements.

\begin{theorem}[Theorems~\ref{prop 1} and \ref{main}]\label{main GM}
Let $R$ be a CM normal domain, admitting  a canonical module, and containing a field of characteristic zero.
Then $R$ has a GM if and only if the following conditions hold.
\begin{enumerate}
\item $R$ is $\Q$-Gorenstein.
\item Every Zariski local canonical cover of $R$ is Gorenstein.
\end{enumerate}
If $R$ is $\Q$-CY of index $n$, then there is a bijection $\GM R\simeq\GM_{\Zn}\!S$ for each canonical cover $S$ of $R$.
\end{theorem}

The field restriction is not strictly needed; see Theorem \ref{prop 1} for a general result for the equivalence, and  Theorem \ref{main} for the most general result for the bijection.

It turns out that Gorenstein modifications induce various categorical equivalences. A typical case is given by the following, which follows from a more general result proved in Theorem \ref{main equivalence text}.

\begin{theorem}[Theorem \ref{main equivalence text}]\label{intro equiv}
Let $R$ be a CM local normal domain with a canonical module, which is $\Q$-Gorenstein of index $n$. Let $S$ be a canonical cover of $R$. Then there is an equivalence of categories
\[
\CM^{\Zn}\!S
\simeq\{X\in\CM R\mid\forall i\in\Z\  (\omega_R^i\otimes_RX)^{**}\in\CM R\}.
\]
\end{theorem}

\subsection{Applications}
We give three main applications.  First, we are able to complement the work of \v{S}penko--Van den Bergh \cite{SpV}, and show that quotient singularities need not have NCCRs.
Recall that a reductive group $G$ acts \emph{generically} on a smooth affine variety $X$ if $X$ contains a point with closed orbit and trivial stabilizer, and the subset consisting of such points has codimension at least two in $X$.

\begin{corollary}[Corollary \ref{SVdB question}]
Let $G$ be a connected reductive group, and $W$ a $G$-representation. If $G$ acts generically on $X\colonequals \Spec SW$ such that $R\colonequals k[X]^G$ is not Gorenstein, then $R$ is not $\Q$-Gorenstein. Therefore $R$ admits no GM, whence no NCCR. \end{corollary}

Second, recall that a CM ring $R$ is called \emph{Cohen-Macaulay-finite} (CM-finite) if the category $\CM R$ admits an additive generator. If $R$ is complete local, this is equivalent to there being only finitely many isomorphism classes of indecomposable CM $R$-modules. By Auslander's argument (see e.g.\ \cite[Theorem 1.6]{NCCR lectures}), if  $\dim R=2$ and $R$ is either a normal domain, or a Gorenstein ring, then NCCRs of $R$ are precisely the additive generators of the category $\CM R$. Therefore, in this case, $R$ has an NCCR if and only if it is CM-finite.

There are two well known results on CM-finite rings.
One is a classification of CM-finite normal domains in dimension two due to Auslander \cite{A3} and Esnault \cite{E}, and the other is a classification of CM-finite Gorenstein rings in arbitrary dimension due to Buchweitz--Greuel--Schreyer \cite{BGS}.
 Theorem \ref{main} enables us to deduce the following Auslander--Esnault classification from the Buchweitz--Greuel--Schreyer classification.

\begin{corollary}\label{Auslander-Esnault}
Let $R$ be normal CM domain of dimension two satisfying one of the following conditions:
\begin{enumerate}
\item[$\bullet$] $R$ is local with algebraically closed residue field of characteristic zero.
\item[$\bullet$] $R$ is finitely generated over an algebraically closed field of characteristic zero.
\end{enumerate}
If $R$ is CM-finite, then complete locally it is a quotient singularity by a finite group.
\end{corollary}

Our third application of  Theorem~\ref{main} is the following, which follows quickly from our main result, together with results of Kawamata \cite{K} and Stafford--Van den Bergh \cite{SV}.  This result is originally due to Ingalls--Yasuda \cite{IY}, who prove it by directly constructing a noncommutative-type canonical cover and considering ramification data, instead of lifting NCCRs to (commutative) canonical covers as we do here.

\begin{corollary}\label{log-terminal}
Assume that $R$ is a normal CM domain which is finitely generated over an algebraically closed field of characteristic zero. 
If $R$ has an NCCR, then it has at worst  log-terminal singularities.
\end{corollary}

\subsection{Notation}
All modules are left modules unless stated otherwise.
For a ring $A$, write $\Mod A$ for the category of $A$-modules, and  $\mod A$ for the full subcategory of $\Mod A$ consisting of finitely generated modules. When $A$ is a $G$-graded ring for a group $G$, we let $\Mod^G\!A$ denote the category of $G$-graded $A$-modules, and $\mod^G\!A$ the full subcategory of $\Mod^G\!A$ consisting of finitely generated modules.  Similarly, $\refl^G\!A$ denotes the full subcategory of $\mod^G\!A$ consisting of reflexive modules. We set $\Hom_A^G(M,N):=\Hom_{\Mod^G\!A}(M,N)$ for $G$-graded $A$-modules $M,N$.

\subsection{Acknowledgements} 
Part of this work was completed during the AIM SQuaRE: Cohen-Macaulay representations and categorical characterizations of singularities.  We thank AIM for funding, and for their kind hospitality.  Dao was further supported by NSA H98230-16-1-0012, Iyama by JSPS Grant-in-Aid for Scientific Research 16H03923, Takahashi by JSPS Grant-in-Aid for Scientific Research 16K05098, and Wemyss by EPSRC grant~EP/K021400/1. We thank Ragnar Buchweitz, Colin Ingalls, Srikanth Iyengar, Yusuke Nakajima, Takehiko Yasuda and Yuji Yoshino for some helpful conversations regarding the material in the paper.

\section{Covering theory}
In order to control the gap between a $\Q$-Gorenstein ring and its canonical covers, we recall the Gabriel cover in Section \ref{Gabriel} below. We give several basic properties, then apply them to obtain Theorem \ref{from S to R} in Section \ref{Gabriel2}, which plays a crucial role in future sections.
We then recall canonical covers of $\Q$-Gorenstein rings in Section~\ref{comm cover}, which are the main objects in this paper. Note that our approach is characteristic independent, and a canonical cover is not necessarily unique.

\subsection{Gabriel covers}\label{Gabriel} Throughout this subsection, let $G$ be a finite abelian group and $A=\bigoplus_{i\in G}A_i$ an arbitrary (not necessarily commutative) $G$-graded ring. We define the \emph{Gabriel cover} of $A$ by
\[ 
\NCov{A}\colonequals (A_{h-g})_{g,h\in G}.
\]
This has a natural ring structure given by the multiplication in $A$ and the matrix multiplication rule, that is,
\[
(x_{g,h})_{g,h\in G}\cdot(y_{g,h})_{g,h\in G}\colonequals (\sum_{i\in G}x_{g,i}y_{i,h})_{g,h\in G}.
\]

\begin{example}
When $G=\Zn$ and $A=\bigoplus_{i=0}^{n-1}A_i$ is a $\Zn$-graded ring, 
\[
\NCov{A}=\left[\begin{array}{ccccc}
A_0&A_1&\cdots&A_{n-2}&A_{n-1}\\
A_{n-1}&A_0&\cdots&A_{n-3}&A_{n-2}\\
\vdots&\vdots&\ddots&\vdots&\vdots\\
A_2&A_3&\cdots&A_0&A_1\\
A_1&A_2&\cdots&A_{n-1}&A_0
\end{array}\right].
\]
\end{example}

Recall that the endomorphism ring $\End_A(N)$ of $N\in\Mod^G\!A$ is canonically $G$-graded by $\End_A(N)_g\colonequals\Hom_A^G(N,N(g))$ for any $g\in G$. The following are basic observations.

\begin{lemma}\label{describe cyclic covering}
There is an isomorphism of rings $\End_A^G(\bigoplus_{g\in G}A(g))\cong\NCov{A}$. More generally, for any $N\in\Mod^G\!A$, there is an isomorphism of rings
\[
\End_A^G(\bigoplus_{g\in G}N(g))\cong \NCov{\End_A(N)}.
\]
\end{lemma}

We regard $A$ as a subring of $\NCov{A}$ by the diagonal embedding
\[
A\to\NCov{A}.
\]
In particular, any $\NCov{A}$-module $X$ is naturally regarded as a $G$-graded $A$-module with $X_g\colonequals e_gX$ for $g\in G$, where $e_g$ is the idempotent of $\NCov{A}$ with $(g,g)$-entry $1$ and the other entries $0$. Consequently, there is a functor
\[
\Mod\NCov{A}\to\Mod^G\!A
\]
called the \emph{push down functor} \cite{G,As}.

\begin{proposition}\label{global 1}
Let $G$ be a finite abelian group and $A$ a $G$-graded ring.
\begin{enumerate}
\item\label{global 1 1}  The push down functor $\Mod\NCov{A}\to\Mod^G\!A$ is an equivalence.
\item\label{global 1 2}  We have $\gl\NCov{A}\le\gl A$ and $\injdim_{\sNCov{A}}\NCov{A}\le\injdim_AA$.
\item\label{global 1 3} As a left (respectively, right) $A$-module, $\NCov{A}$ is free of rank $\order{G}$. Therefore the equivalence $\Mod\NCov{A}\to\Mod^G\!A$ sends $\NCov{A}$ to $A^{\oplus \order{G}}$.
\end{enumerate}
\end{proposition}

\begin{proof}
(1) This is well known, see e.g. \cite[Theorem 3.1]{IL}.\\
(3) Since $(\NCov{A})e_g\cong A$ as left $A$-modules, the decomposition $\NCov{A}=\bigoplus_{g\in G}(\NCov{A})e_g$ as $A$-modules gives the result.\\
(2) Since $\Ext^i_{\Mod^G\!A}(X,Y)=\Ext^i_A(X,Y)_0$ holds for all $X,Y\in\Mod^G\!A$, it follows that
\begin{eqnarray*}
&\gl\NCov{A}=\gl(\Mod^G\!A)\le\gl A,&\\
&\injdim_{\sNCov{A}}\NCov{A}=\injdim_{\Mod^G\!A}A\le\injdim_AA&
\end{eqnarray*}
by \eqref{global 1 1} and \eqref{global 1 3}.
\end{proof}

Now we study when equalities in Proposition \ref{global 1}\eqref{global 1 2} hold. Recall that an additive functor $F\colon\CC\to\DD$ is \emph{dense up to direct summands} if for any $X\in\DD$, there exists $Y\in\CC$ such that $X$ is a direct summand of $F(Y)$.

\begin{proposition}\label{gl.dim of *}
Let $G$ be a finite abelian group, and $A$ a $G$-graded ring such that $\order{G}$ is invertible in $A_0$.
\begin{enumerate}
\item\label{gl.dim of * 1} The forgetful functor $\Mod^G\!A\to\Mod A$ is dense up to direct summands.
\item\label{gl.dim of * 2} We have $\gl\NCov{A}=\gl A$ and $\injdim_{\sNCov{A}}\NCov{A}=\injdim_AA$.
\item\label{gl.dim of * 3} $A$ is a direct summand of $\NCov{A}$ as an $A$-bimodule.
\end{enumerate}
\end{proposition}

\begin{proof}
(3) We claim that the diagonal embedding $A\to\NCov{A}$ splits, as a morphism of $A$-bimodules.  Since $\order{G}$ is invertible, the morphism
\[
(x_{g,h})_{g,h\in G}\mapsto \left(\frac{1}{\order{G}}\sum_{h-g=i}x_{g,h}\right)_{i\in G}
\]
gives a splitting.\\
(1) For any $A$-module $X$, consider the  $\NCov{A}$-module $Y\colonequals (\NCov{A})\otimes_AX$.
Then $X$ is a direct summand of $Y$ as an $A$-module by \eqref{gl.dim of * 3}.\\
(2) Since $\bigoplus_{g\in G}\Ext^i_{\Mod^G\!A}(X,Y(g))=\Ext^i_A(X,Y)$ holds  for any $X,Y\in\Mod^G\!A$, the assertion follows from \eqref{gl.dim of * 1} by a similar argument as in the proof of Proposition \ref{global 1}\eqref{global 1 2}.
\end{proof}

In Proposition \ref{gl.dim of *}, we can not drop the assumption that $\order{G}$ is invertible in $A_0$.

\begin{example}
Let $k$ be a field of characteristic $2$ and $A\colonequals k[x]/(1+x^2)$ be a $\Z_2$-graded ring given by $\deg x=1$.
Then we have $\gl A=\infty$ since $A=k[x]/(1+x)^2$.
On the other hand, $A\mypound\Z_2={\rm M}_2(k)$ holds, and hence we have $\gl A\mypound\Z_2=0$.
\end{example}

Now we show that there is a close relationship between Gabriel covers and skew group rings under the assumption that the centre of $A_0$ contains a primitive $\order{G}$-th root $\zeta$ of unity.
Let $\langle\zeta\rangle$ be the subgroup of $A_0^\times$ generated by $\zeta$ and $G^\vee:=\Hom_{\Z}(G,\langle\zeta\rangle)$ be the dual abelian group of $G$. Then there is a non-degenerate pairing $G^\vee\times G\to\langle\zeta\rangle$, which gives an action of $G^\vee$ on $A$ by
\begin{equation}\label{G action}
f\cdot (x_g)_{g\in G}=(f(g)x_g)_{g\in G}.
\end{equation}
Hence we may form the skew group ring $A*G^\vee$. 
For an $A$-module $X$ and $f\in G^\vee$, we write ${}_fX$ for the $A$-module obtained by twisting the action of $A$ on $X$ by $f$.
Similarly we define $Y_f$ for an $A^{\op}$-module $Y$ and $f\in G^\vee$.

We note the following information, although we will not use it in this paper.

\begin{proposition}\label{G-covering is *}
Let $G$ be a finite abelian group and $A$ a $G$-graded ring such that the centre of $A_0$ contains $\order{G}^{-1}$ and a primitive $\order{G}$-th root of unity.
\begin{enumerate}
\item As rings, $\NCov{A}\cong A*G^\vee$.
\item As $A$-bimodules, $\NCov{A}\cong\bigoplus_{f\in G^\vee}A_f$.
\item For any $X\in\Mod A$, there exists $Y\in\Mod^G\!A$ such that $Y\cong\bigoplus_{f\in G^\vee} {}_fX$ as $A$-modules.
\end{enumerate}
\end{proposition}

\begin{proof}
(1) We regard $G^\vee$ as a subgroup of $(\NCov{A})^\times$ by
\[
G^\vee\ni f\mapsto (\delta_{g,h}f(g))_{g,h\in G},
\]
where $\delta_{g,h}$ is the Kronecker delta. Then the conjugate $f^{-1}\cdot f$ restricts to an automorphism of the subring $A$ of $\NCov{A}$, which coincides with the action \eqref{G action} since the $(g,h)$-entry of $f^{-1}xf$ is $f^{-1}(g)x_{h-g}f(h)=f(h-g)x_{h-g}$.
Thus the above embeddings give a morphism $A*G^\vee\to \NCov{A}$ of rings.
This is bijective since there is a direct sum decomposition $\NCov{A}=\bigoplus_{f\in G^\vee}Af$, which 
can be checked by using the Vandermonde determinant.\\
(2) We have isomorphisms $\NCov{A}=\bigoplus_{f\in G^\vee}Af$ and $Af\cong A_f$ of $A$-bimodules.\\
(3) Let $Y:=(\NCov{A})\otimes_AX\in\Mod\NCov{A}=\Mod^G\!A$. Then in $\Mod A$ 
\[
Y\cong\bigoplus_{f\in G^\vee}A_f\otimes_AX=\bigoplus_{f\in G^\vee}{}_fX.\qedhere
\]
\end{proof}

\subsection{Application to Gorenstein modifications}\label{Gabriel2}

In this subsection, we apply observations in the previous subsection to study Gorenstein modifications.

We first recall a basic notion. For a commutative ring $R$, we say that an $R$-algebra $A$ is \emph{symmetric} if $\Hom_R(A,R)\cong A$ as $A$-bimodules. 

\begin{lemma}\label{S R reflexive}
Let $R$ be a commutative ring and $A$ a symmetric $R$-algebra.
\begin{enumerate}
\item\label{S R reflexive 1} There is an isomorphism $\Hom_{R}(-,R)\cong\Hom_A(-,A)$ of functors $\Mod A\to\Mod A^{\op}$.
\item\label{S R reflexive 2} An $A$-module is reflexive if and only if it is reflexive as an $R$-module.
\end{enumerate}
\end{lemma}
\begin{proof}
(1) The given isomorphism $A\cong\Hom_{R}(A,R)$ induces  isomorphisms of functors 
\[
\Hom_{A}(-,A)\cong\Hom_A(-,\Hom_{R}(A,R))\cong\Hom_R(-,R).
\]
(2) By \eqref{S R reflexive 1}, there is an isomorphism $\Hom_R(\Hom_{R}(-,R),R)\cong\Hom_{A^{\op}}(\Hom_A(-,A),A)$ of functors $\Mod A\to\Mod A$. It can be checked easily that this commutes with the evaluations $1_{\Mod A}\to\Hom_R(\Hom_{R}(-,R),R)$ and $1_{\Mod A}\to\Hom_{A^{\op}}(\Hom_{A}(-,A),A)$, and thus the assertion follows.
\end{proof}

Recall that a $G$-graded ring $A$ is \emph{strongly $G$-graded} if the natural map
\[
A_h\to\Hom_{A_0}(A_g,A_{g+h}),\ \ \ x\mapsto(y\mapsto yx)
\]
is an isomorphism for any $g,h\in G$. We need the following easy observations.

\begin{lemma}\label{End is G-covering}
Let $G$ be a finite abelian group, and $A$ a $G$-graded ring. 
\begin{enumerate}
\item \label{End is G-covering 1} 
There is a natural morphism
\begin{equation*}
\upvarphi\colon\NCov{A}\to\End_{A_0}(A)
\end{equation*}
of rings sending $(x_{g,h})_{g,h\in G}$ to the morphism $(y_h)_{h\in G}\mapsto (\sum_{g\in G}y_gx_{g,h})_{h\in G}$.
\item\label{End is G-covering 2} $A$ is strongly $G$-graded if and only if $\upvarphi$ is an isomorphism.
\item\label{End is G-covering 3} If $A$ is strongly $G$-graded, then for any $g\in G$, the morphism $A\to\Hom_{A_0}(A,A_g)$ sending $x\in A$ to $y\mapsto (yx)_g$ is an isomorphism of $A$-modules.
\item\label{End is G-covering 4} If $A$ is strongly $G$-graded and commutative, then $A$ is a symmetric $A_0$-algebra.
\end{enumerate}
\end{lemma}

\begin{proof}
(1)(2)(3) Immediate from definition.\\
(4) Since $A$ is commutative, the isomorphism $A\cong\Hom_{R}(A,R)$ in \eqref{End is G-covering 3} for $g=0$ is an isomorphism of $A$-bimodules.
\end{proof}

For a symmetric $R$-algebra $A$, we denote by $\refl A$ the category of reflexive $A$-modules, or equivalently, $A$-modules which are reflexive as an $R$-module by Lemma \ref{S R reflexive}\eqref{S R reflexive 2}.

We now prepare results on graded normal domains.
First we recall the following well-known result; see e.g.\ \cite[Proposition 2.4]{IR}.

\begin{proposition}\label{ref eq}
Let $R$ be a normal domain and $M\in\refl R$. Then $\End_R(M)$ is a symmetric $R$-algebra and there is an equivalence
\[\Hom_R(M,-):\refl R\to\refl\End_R(M).\]
\end{proposition}

Using this gives the following key observation.

\begin{proposition}\label{E and Sn}
Let $G$ be a finite abelian group, and $S$ a commutative strongly $G$-graded ring such that $R\colonequals S_0$ is a normal domain, and $S$ is a finitely generated $R$-module.  
\begin{enumerate}
\item\label{E and Sn 0} If $X\in\mod S$, then $X\in\CM S$ if and only if $X\in\CM R$.
\item\label{E and Sn 1} The functor $(-)_0\colon\refl^G\!S\to\refl R$ is an equivalence.
\item\label{E and Sn 2} If $N\in\refl^G\!S$, then there is an isomorphism of rings $\NCov{\End_S(N)}\cong\End_R(N)$, and an equivalence $\Mod^G\!\End_S(N)\simeq\Mod\End_R(N)$.
\item\label{E and Sn 3} If $N\in\refl^G\!S$, then there are inequalities 
\begin{align*}
\gl\End_R(N)&\le\gl\End_S(N)\\
\injdim_{\End_R(N)}\End_R(N)&\le\injdim_{\End_S(N)}\End_S(N).
\end{align*}
These are equalities if $R$ contains $\order{G}^{-1}$.
\end{enumerate}
\end{proposition}

\begin{proof}
(1) This is well known, see for example \cite[Exercise 1.2.26]{BH}.\\
(2) There are equivalences
\[
\Mod^G\!S
\stackrel{\scriptstyle\ref{global 1}\eqref{global 1 1}}{\simeq}\Mod(\NCov{S})
\stackrel{\scriptstyle\ref{End is G-covering}\eqref{End is G-covering 2}}{\simeq}\Mod\End_R(S).
\]
Since both $S$ and $\End_R(S)$ are symmetric $R$-algebras by Lemmas \ref{End is G-covering}\eqref{End is G-covering 4} and \ref{ref eq}, there is  an induced equivalence $\refl^G\!S\simeq\refl\End_R(S)$ by Lemma \ref{S R reflexive}\eqref{S R reflexive 2}. On the other hand, by Proposition \ref{ref eq}, there is an equivalence $\Hom_R(S,-)\colon\refl R\to\refl\End_R(S)$, whose inverse is given by $e(-)\colon\refl\End_R(S)\to\refl R$ for the idempotent $e$ corresponding to the summand $R$ of $S$.
Composing them gives equivalences $\refl^G\!S\simeq\refl\End_R(S)\simeq\refl R$.
It is easy to check that the composition coincides with $(-)_0$.\\
(3) Let $L\colonequals \bigoplus_{i\in G}N(i)\in\refl^G\!S$. Then as an $R$-module, $L_0=\bigoplus_{i\in G}N_i=N\in\refl R$, so it follows that 
$\End_R(N)\stackrel{\eqref{E and Sn 1}}{=}\End_S^G(L)\stackrel{\scriptstyle{\ref{describe cyclic covering}}}{=}\NCov{\End_S(N)}$, and so there is an equivalence $\Mod^G\!\End_S(N)\simeq\Mod\End_R(N)$ by Proposition \ref{global 1}\eqref{global 1 1}.\\
(4) The assertions follow from Propositions \ref{global 1}\eqref{global 1 2} and \ref{gl.dim of *}\eqref{gl.dim of * 2}.
\end{proof}

The following is the main result of this subsection.

\begin{theorem}\label{from S to R}
Let $G$ be a finite abelian group, $S$ a commutative strongly $G$-graded ring such that $R\colonequals S_0$ is a normal domain, $S$ is a finitely generated $R$-module and $N\in\refl^G\!S$. 
\begin{enumerate}
\item\label{from S to R 1} $N$ gives a modification of $S$ if and only if it gives a modification of $R$.
\item\label{from S to R 2} If $N$ gives a GM (respectively, NCCR, NCR) of $S$, then it gives a GM (respectively, NCCR, NCR) of $R$. The converse holds if $R$ contains $|G|^{-1}$.
\end{enumerate}
\end{theorem}
\begin{proof}
(1) First, by Lemmas \ref{S R reflexive}\eqref{S R reflexive 2} and \ref{End is G-covering}\eqref{End is G-covering 4}, $N\in\refl R$.  The result then follows since
\[
\End_S(N)\in\CM S 
\stackrel{\scriptstyle\ref{E and Sn}\eqref{E and Sn 0}}{\iff}
\End_S(N)\in\CM R
\stackrel{\scriptstyle\ref{E and Sn}\eqref{E and Sn 2}}{\iff}
\End_R(N)\in\CM R.
\]
(2) Since GM, NCCR and NCR are local properties, we can assume that $R$ is local.

Suppose that $N$ gives a GM of $S$.  Then by \eqref{from S to R 1}, $\End_R(N)\in\CM R$ holds.
Hence by Proposition \ref{E and Sn}\eqref{E and Sn 3} and the assumption that $N\in\GM S$ we see the desired inequality
\[
\injdim_{\End_R(N)}\End_R(N)\leq \injdim_{\End_S(N)}\End_S(N)=\dim S.
\]

Similarly, if $N$ gives an NCCR of $S$, then again by \eqref{from S to R 1} $\End_R(N)\in\CM R$ holds.
Hence by Proposition \ref{E and Sn}\eqref{E and Sn 3} and the assumption that $N\in\NCCR S$ we see the desired inequality
\[
\gl\End_R(N)\leq \gl\End_S(N)=\dim S.
\]

If $N$ gives an NCR of $S$, then the inequality $\gl\End_R(N)\leq \gl\End_S(N)$ from Proposition \ref{E and Sn}\eqref{E and Sn 3}  immediately shows that $N$ gives an NCR of $R$.

The converse to all three statements is obvious, since the inequalities in Proposition \ref{E and Sn}\eqref{E and Sn 3} are equalities when $R$ contains $|G|^{-1}$.
\end{proof}

\subsection{Canonical covers}\label{comm cover}
 In this subsection we briefly recall canonical covers, which is slightly subtle in our general setting since they need not be unique.

Let $R$ be a normal domain with quotient field $K$.
For a finite abelian group $G$ and a group homomorphism $\upsigma\colon G\to\Cl(R)$, where $\Cl(R)$ is the class group, this subsection constructs a strongly $G$-graded ring $S$ such that $S_0=R$ and $[S_i]=\upsigma(i)$ in $\Cl(R)$ for all $i\in G$.

We start by fixing a decomposition $G=\bigoplus_{a=1}^\ell\Z_{n_a}$.
For $a=1,\ldots,\ell$, we then choose a divisorial ideal $I_a$ of $R$ corresponding to the element $(0,\ldots,1,\ldots,0)\in G$.
Moreover, we fix $q=(q_1,\ldots,q_\ell)\in (K^\times)^\ell$, where each $q_a$ is chosen such that the multiplication map
\[
q_a\colon(I_a^{n_a})^{**}\cong R
\]
is an isomorphism. For $i=(i_1,\ldots,i_\ell)\in\Z^{\ell}$, we will write
\[
S_i\colonequals (I_1^{i_1}\cdots I_\ell^{i_\ell})^{**},
\]
which is a divisorial ideal of $R$. Using the natural map $\Z^\ell\to G$, we identify $G$ with a subset $J$ of $\Z^\ell$, namely
\begin{equation}\label{identify}
J\colonequals \prod_{a=1}^\ell\{0,1,2,\ldots,n_a-1\}\cong G,
\end{equation}
and consider the $J$-graded (and hence $G$-graded) $R$-module
\[
S\colonequals \bigoplus_{i\in J}S_i.
\]
Using the isomorphisms $q_a$, we will next define a multiplication on $S$, which makes $S$ into a $G$-graded ring.  To fix notation, for $i,j\in\Z^\ell$, let
\[
q(i)\colonequals \prod_{a=1}^\ell q_a^{\lfloor\frac{i_a}{n_a}\rfloor}\in K^\times\ \mbox{ and }\ 
q(i,j)\colonequals \frac{q(i+j)}{q(i)q(j)}\in K^\times.
\]
Then for $x\in S_i$ and $y\in S_j$, the product $x\cdot_q y$ is defined by
\begin{equation}\label{product}
x\cdot_qy\colonequals q(i,j)xy\in K,
\end{equation}
where the right hand side is the multiplication in $K$. Then this belongs to $S_{(i+j)_J}$, where $(i+j)_J$ is the element of $J$ corresponding to $i+j\in G$ by the bijection \eqref{identify}.
In fact $xy\in I^{i+j}$ holds, and therefore we have $q(i,j)xy\in I^{(i+j)_J}=S_{i+j}$.

The formula \eqref{product} gives a multiplication on $S$. In fact we only have to check associativity, which follows from a general formula
\[
q(i,j)q(i+j,k)=q(i,j+k)q(j,k)
\]
for all $i,j,k\in \Z^\ell$.

The following is the main result in this subsection, which is used also in \cite{IN}.

\begin{proposition}\label{construct S}
Let $R$ be a normal domain, $\upsigma\colon G\to\Cl(R)$ a group homomorphism where $G$ is a finite abelian group, and $S$ the ring constructed above. Then the following statements hold.
\begin{enumerate}
\item\label{construct S 1} $S$ is strongly $G$-graded, with $S_0=R$ and $[S_i]=\upsigma(i)$ in $\Cl(R)$ for all $i\in G$.
\item\label{construct S 2} There is an isomorphism $\End_R(S)\cong \NCov{S}$, as $R$-algebras.
\item\label{construct S 3} For any $i\in G$, there is an isomorphism $\Hom_R(S,S_i)\cong S$ of $S$-modules.
\item\label{construct S 4} If $R$ is local (respectively, complete local) and $\upsigma$ is injective, then $S$ is local (respectively, complete local).
\end{enumerate}
\end{proposition}

\begin{proof}
(1) By our construction, $(S_i\cdot_q S_j)^{**}=S_{i+j}$ holds for all $i,j\in G$. Thus the map $S_j\to\Hom_R(S_i,S_{i+j})$ sending $x\in S_j$ to $(y\mapsto yx)$ is an isomorphism.\\
(2) This is immediate from \eqref{construct S 1} and Lemma \ref{End is G-covering}\eqref{End is G-covering 2}.\\
(3) This is Lemma \ref{End is G-covering}\eqref{End is G-covering 3}.\\
(4) Assume that $(R,\mm)$ is a local ring.
We first show that $S$ is a local ring with  maximal ideal $\nn\coloneqq\mm\oplus(\bigoplus_{0\neq i\in G}S_i)$. 
Clearly $\nn$ is a maximal ideal. Since $S_iS_{-i}\subset\mm$ holds for any $0\neq i\in G$, it follows that some power of $\nn$ is contained in $\mm S$ and hence in all maximal ideals of $S$. Therefore $\nn$ is a unique maximal ideal of $S$.
If $(R,\mm)$ is complete, then $S$ is complete with respect to the $\mm S$-adic topology, and hence it is also complete with respect to the $\nn$-adic topology.
\end{proof}

Since the ring structure of $\End_R(S)$ uses only the $R$-module structure of $S$, it follows from Proposition \ref{construct S}\eqref{construct S 2} that the ring structure of $\NCov{S}$ does not depend on the choice of an element $q\in(K^\times)^\ell$.

\begin{definition}\label{cancov}
When $R$ is $\mathds{Q}$-CY, consider the group homomorphism $\upsigma\colon \Zn\to\Cl(R)$ defined by $i\mapsto\omega_R^i$. Consider an embedding $\omega_R\to R$ and an isomorphism $q\colon(\omega_R^n)^{**}\cong R$ given by multiplication by an element $q\in K^\times$. 
Applying the above construction gives a \emph{canonical cover} of $R$, defined by
\[
S_q=S=\bigoplus_{i=0}^{n-1}\omega_R^i.
\]
\end{definition}

Note that there is a freedom of the choice of $q$ up to an element in $R^\times$. In general, $S_q$ depends on the choice of $q$, see \cite[Example 1.6]{TW}. We note the following, although we do not use it in this paper.

\begin{lemma}
With the assumptions as above, the following hold.
\begin{enumerate}
\item\label{indep 1} If $q=r^nq'$ for some $r\in R^\times$, then $S_q\cong S_{q'}$ as $\Zn$-graded $R$-algebras.
\item If $R$ is a complete local ring with algebraically closed residue field of characteristic zero, then $S_q$ is independent of the choice of $q$.
\end{enumerate}
\end{lemma}
\begin{proof}
(1) It is easy to check that the map
\begin{eqnarray*}
S_q&\xrightarrow{}&S_{q'}\\
(a_i)_{0\le i<n}&\mapsto&(a_ir^i)_{0\le i<n}
\end{eqnarray*}
gives the desired isomorphism.\\
(2) We need to show that $qq'^{-1}\in R^\times$ has an $n$-th root in $R^\times$. Since the residue field is algebraically closed, we can assume that $qq'^{-1}=1+r$ for some $r\in\mm$. Then it has an $n$-th root given by the Taylor expansion since $R$ contains $\Q$.
\end{proof}

The following properties are well known.

\begin{proposition}\label{gorensteiness}
Let $R$ be a CM normal domain which is $\Q$-CY of index $n$, and $S$ a canonical cover of $R$.
\begin{enumerate}
\item\label{gorensteiness 1} $S$ is a Gorenstein ring if and only if $S\in\CM R$.
\item\label{gorensteiness 2} If $R$ contains a field of characteristic zero, then $S$ is a normal domain.
\end{enumerate}
\end{proposition}

\begin{proof}
(1) We may assume that $R$ is local, whence so is $S$ by Proposition \ref{construct S}\eqref{construct S 4}. By Proposition \ref{E and Sn}\eqref{E and Sn 0}, we have that $S$ is a CM $R$-module if and only if $S$ is a CM ring. Further $\Hom_R(S,\omega_R)$ is a canonical module of $S$ \cite[Theorem 3.3.7]{BH}, and this is isomorphic to $S$ by Proposition \ref{construct S}\eqref{construct S 3}.\\
(2) This is \cite[Proposition 1.12]{TW}.
\end{proof}


\section{Gorenstein Modifications and $\mathds{Q}$-Gorenstein Rings}

In this section we develop the notion of Gorenstein modifications, and relate these (and NCCRs) to $\mathds{Q}$-Gorenstein rings.  The main applications are given in Section \ref{app section}.

\subsection{Determinant}
When $R$ is a normal domain, recall that the determinant of a finitely generated module $M$ of positive rank $n$, denoted by $\det M$, is given by $\bigwedge^n\! M$.
It represents an element in the class group of $R$. As usual, we also write $[M]$ for the attached divisor of $M$, see  \cite[VII \textsection 4]{Bour}.

\begin{proposition}\label{det}
Suppose $M,N\in\mod R$ have positive ranks such that $M$ is torsion-free. Then
\begin{enumerate}
\item\label{det 1} $[\det(M\otimes_R N)] = (\rk M)[\det N] + (\rk N)[\det M]$.
\item\label{det 2} $[\det\Hom_R(M, N)] = (\rk M)[\det N] - (\rk N)[\det M]$.
\item\label{det 3} $[\det\End_R(M)] = 0$.
\end{enumerate}
\end{proposition}
\begin{proof}
This is \cite[VII, \textsection 4, Exercise 8(d)]{Bour}. We give a proof for convenience of the reader.\\
(1)  There is a short exact sequence
\[
0 \to F \to N \to N' \to 0 
\]
where $F$ is free of same rank as $N$, and $N'$ is torsion. Tensoring  with $M$, noting that $\Tor_1^R(M, N')$ is torsion and $M\otimes_RF$ is either $0$ or torsion-free, gives an exact sequence
\[
0 \to M\otimes_R F \to  M\otimes_RN \to M\otimes_RN' \to 0
\]
The determinant of the leftmost module is $(\rk N)[\det M]$. For the rightmost  module, note that the determinant of a torsion module $T$ can be calculated by
\[
[\det T] =\sum_{\hei \pp=1} \len_{R_{\pp}}(T_{\pp})[\pp] 
\]
When one localizes at any such primes, the length of $(M\otimes N')_{\pp}$ would be exactly $(\rk M)(\len{N'_{\pp}})$, so $[\det(M\otimes_R N')] = (\rk M)[\det N'] = (\rk M)[\det N]$. \\
(2) The map $M^*\otimes_R N \to \Hom_R(M,N)$ is an isomorphism in height one, so the modules have the same determinant. Thus it suffices to show that $[\det M^*] = -[\det M]$.  Dualizing the exact sequence
$0 \to F \to M \to M' \to 0$ with a free module $F$ and a torsion module $M'$ gives an exact sequence
\[
0 \to M^* \to F^* \to \Ext_R^1(M',R) \to \Ext_R^1(M,R) \to 0.
\]
As $\Ext_R^1(M,R)$ is supported in height at least $2$, we have \[
[\det(M^*)] = -[\det\Ext_R^1(M',R)] =-[\det M'] =-[\det M],
\]
where the second equality is proved by localizing at height one primes.\\
(3) Immediate from \eqref{det 2}.
\end{proof}

The above proposition fails if we drop the assumption that $M$ is torsion-free.

\subsection{NCCRs and $\mathds{Q}$-Gorenstein}\label{Sect 3.2}

Recall from the introduction that for $M\in\refl R$ with $A:=\End_R(M)$, $A$ is a \emph{modification of $R$} if $A\in\CM R$, and a \emph{Gorenstein modification of $R$} if $A\in\CM R$ and $\injdim_{A\otimes_RR_\pp}(A\otimes_RR_{\pp})=\dim R_{\pp}$ for all $\pp\in\Supp_RA$. It is clear that giving an NCCR is equivalent to giving GM and NCR at the same time. Therefore the following implications hold
\[
\mbox{NCR}\Longleftarrow
\mbox{NCCR}
\Longleftrightarrow
\mbox{GM+NCR}
\Longrightarrow
\mbox{GM}
\Longrightarrow\\
\mbox{Modification},
\]
where in general all implications are strict.  However,  if $R$ is a Gorenstein normal domain, then any modification of $R$ is GM \cite[Proposition 2.21, Lemma 2.22]{IW1}. 

With reasonable conditions on $R$, we record in the next lemma known alternative characterisations.

\begin{lemma}\label{Goren char}
Suppose that $R$ is a CM ring with a canonical module $\omega_R$, $M\in\refl R$ and $A:=\End_R(M)$.  Then the following conditions are equivalent.
\begin{enumerate}
\item $A$ is a Gorenstein modification of $R$.
\item $A\in\CM R$ and $\Hom_R(A,\omega_R)\in\proj A$.
\item $A\in\CM R$ and $\Hom_R(A,\omega_R)\in\proj A^{\op}$.
\item $A^{\op}$ is a Gorenstein modification of $R$.
\end{enumerate}
\end{lemma}
\begin{proof}
(1)$\Leftrightarrow$(2), and (3)$\Leftrightarrow$(4), is \cite[Proposition 1.1(3)]{GN}, whereas (2)$\Leftrightarrow$(3) is \cite[Lemma 2.15]{IW1}.
\end{proof}

Leading up to our next result, we require the following well known property of the Nakayama functor.

\begin{lemma}\label{AR type}
Suppose that $R$ is a normal CM domain with a canonical module $\omega_R$. Then there are functorial isomorphisms
\[
\Hom_R(Y,\upnu(X))\cong\Hom_R(X,Y)^{\vee}\cong\Hom_R(\upnu^{-1}(Y),X)
\]
for all $X,Y\in\refl R$.
\end{lemma}
\begin{proof}
The first isomorphism follows from
\[
\Hom_R(Y,\upnu(X))\cong(X^*\otimes_RY)^{\vee}\cong\Hom_R(X,Y)^{\vee}.
\]
Since $\nu$ is an autoequivalence of $\refl R$, we have $\Hom_R(Y,\upnu(X))\cong\Hom_R(\upnu^{-1}(Y),X)$.
\end{proof}

\begin{proposition}\label{prop 0 text}
Let $R$ be a CM normal domain with a canonical module $\omega_R$.
Assume that $M\in\refl R$ gives a modification of $R$. Then $M$ gives a GM of $R$ if and only if $\add M=\add \upnu(M)$ holds. In this case, $\upnu$ gives an autoequivalence on $\add M$.
\end{proposition}
\begin{proof}
By our assumption, $\Lambda\colonequals \End_R(M)$ is a CM $R$-module. By Proposition \ref{ref eq}, there is an equivalence $F\colonequals \Hom_R(M,-)\colon\refl R\to\refl\Lambda$.
By Lemma~\ref{AR type}, $\Lambda^\vee=\End_R(M)^\vee=\Hom_R(M,\upnu(M))=F(\upnu(M))$ holds.  By \cite[Lemma 2.15]{IW1}  and as in Lemma \ref{Goren char}, $\Lambda$ is a Gorenstein modification of $R$ if and only if $\add_\Lambda\Lambda=\add_\Lambda(\Lambda^\vee)$ if and only if $\add_\Lambda F(M)=\add_\Lambda F(\upnu(M))$ if and only if $\add_RM=\add_R\upnu(M)$.
In this case, the autoequivalence $\upnu\colon\refl R\to\refl R$ restricts to an autoequivalence $\add M\to\add M$.
\end{proof}

\begin{theorem}\label{prop 1}
Let $R$ be a CM normal domain with a canonical module $\omega_R$.
\begin{enumerate}
\item\label{prop 1 1} If $R$ is $\Q$-CY and its canonical cover $S$ is Gorenstein, then $S$ gives a GM of $R$. 
More generally, assume that $R$ is $\Q$-Gorenstein of index $n$, and for any affine open subset $\Spec U$ of $\Spec R$ which is $\Q$-CY, a canonical cover of $U$ is Gorenstein. Then $\bigoplus_{i=0}^{n-1}\omega_R^i$ gives a GM of $R$.
\item\label{prop 1 2} If $M\in\refl R$ gives a GM of $R$, then $R$ is $\Q$-Gorenstein. Moreover, if $\rk M$ is invertible in $R$, then for any affine open subset $\Spec U$ of $\Spec R$ which is $\Q$-CY, any canonical cover of $U$ is Gorenstein.
\end{enumerate}
\end{theorem}
\begin{proof}
(1) We have only to show the first statement since giving a GM is a local property.
Assume that $R$ is $\Q$-CY and $S=\bigoplus_{i=0}^{n-1}\omega_R^i$ is a canonical cover of $R$.
Then $\End_R(S)\cong\NCov{S}$ is a CM $R$-module.
Since $\upnu(S)\cong S$ as $R$-modules, it follows from Proposition \ref{prop 0 text} that $S$ gives a GM of $R$.\\
(2) For the first statement, since projective is a local property, it follows that $\Q$-Gorenstein is a local property, and so it suffices to show that $R_\mm$ is $\Q$-CY for each maximal ideal $\mm$ of $R$. Since the class group of $R_\mm$ is a subgroup of the class group of its completion $\widehat{R}$, it suffices to show that $\widehat{R}$ is $\Q$-CY. Since GM is preserved by localization and completion, we can assume that $R$ is complete local without loss of generality.

Assume that a basic $R$-module $M$ gives a GM of $R$. By Proposition \ref{prop 0 text} and the Krull--Schmidt--Azumaya Theorem, we have
\[
M\cong\upnu(M)=(\omega_R\otimes_RM)^{**}.
\]
By Proposition \ref{det}\eqref{det 1}, this gives an equality $[\det M]=[\det M]+(\rk M)[\omega_R]$ in the class group of $R$.
Hence $[\omega_R]$ is torsion in the class group, so $R$ is $\mathds{Q}$-CY.

Now we prove the second statement. By Proposition \ref{gorensteiness}\eqref{gorensteiness 1}, it suffices to show that $\omega_U^i$ is a CM $U$-module for any $i\in\Z$.
By our assumption $(\rk M)^{-1}\in R$ and \cite[Proposition A.2, Corollary A.5]{HL}, $U$ is a direct summand of the $U$-module $\End_U(M)$.
Thus $\omega_U^i$ is a direct summand of the $U$-module $(\End_U(M)\otimes_U\omega_U^i)^{**}$.
Since $(\omega_U^i\otimes_UM)^{**}\in\add M$ holds by Proposition \ref{prop 0 text}, 
\[
(\End_U(M)\otimes_U\omega_U^i)^{**}\cong\Hom_U(M,(\omega_U^i\otimes_UM)^{**})\in\add_U\End_U(M)\subset\CM U.
\]
Consequently, $\omega_U^i$ is a CM $U$-module.
\end{proof}

Note that the first statement in Theorem \ref{prop 1}\eqref{prop 1 2} can also be deduced from the following.

\begin{lemma}\label{3.3}
Let $R$ be a CM normal domain with a canonical module $\omega_R$. Let $M$ be a torsion-free $R$-module and $A=\End_R(M)$. If $A \cong \Hom_R(A,\omega_R)$ as $R$-modules, then $R$ is  $\Q$-CY.
\end{lemma}
\begin{proof}
By Proposition \ref{det}\eqref{det 3}, $[\det A]=0$. Again by Proposition \ref{det}\eqref{det 2},
\[
0=[\det A]=[\det\Hom_R(A,\omega_R)]=(\rk A)[\omega_R] - [\det A]=(\rk A)[\omega_R].\qedhere
\]
\end{proof}

We have the following explicit correspondences between Gorenstein modifications of $R$ and Gorenstein modifications of $S$.

\begin{theorem}\label{main}
Let $R$ be a CM normal domain with a canonical module $\omega_R$ which is $\Q$-CY of index $n$.  Let $S$ be a canonical cover of $R$.
\begin{enumerate}
\item\label{main 1} If $N\in\refl S$ gives a GM (respectively, NCCR, NCR) of $S$ and is $\Z_n$-gradable, then it gives a GM (respectively, NCCR, NCR) of $R$.
\item\label{main 2} If $M\in\refl R$ gives a GM (respectively, NCCR) of $R$ and $R$ contains $n^{-1}$, then $(S\otimes_RM)^{**}$ gives a GM (respectively, NCCR) of $S$.
\item\label{main 3} If $R$ contains $n^{-1}$, then there are bijections
\[
\GM R\simeq\GM_{\Zn}\!S\ \mbox{ and }\ \NCCR R\simeq\NCCR_{\Zn}\!S.
\]
\end{enumerate}
\end{theorem}
\begin{proof}
(1) This is shown in Theorem \ref{from S to R}\eqref{from S to R 2}. \\
(2) Assume that $M\in\refl R$ gives a GM (respectively, NCCR) of $R$. Let $N\colonequals (S\otimes_RM)^{**}\in\refl^{\Zn}\!S$.
Then $\add_RM=\add_RN$ holds by Proposition \ref{prop 0 text}, and thus $\End_R(M)$ is Morita equivalent to $\End_R(N)$. Therefore $N$ gives a GM (respectively, NCCR) of $R$. Since $R$ contains $n^{-1}$, $N$ gives a GM (respectively, NCCR) of $S$ by Theorem \ref{from S to R}\eqref{from S to R 2}.\\
(3) The assertion is immediate from \eqref{main 1} and \eqref{main 2}.
\end{proof}

We can not add `NCR' to Theorem \ref{main}\eqref{main 2} by the following known observation.

\begin{remark}
Let $R$ be a local rational surface singularity containing an algebraically closed residue field of characteristic zero. 
\begin{enumerate}
\item $R$ has an NCR and is $\Q$-Gorenstein.
\item  If $R$ is not a quotient singularity by a finite group, then no canonical cover $S$ of $R$ is a rational singularity, or
has an NCR. More strongly, there is no faithful $S$-module $M$ satisfying $\gl\End_S(M)<\infty$.
\end{enumerate}
\end{remark}

\begin{proof}
(1) $R$ has an NCR by \cite[Corollary 2.15]{IWclass}. Moreover $R$ is $\Q$-Gorenstein since the class group is finite \cite[Proposition 17.1]{L}.\\
(2) Let $S$ be a canonical cover of $R$ of index $n$, and $\zeta\in R$ the primitive $n$-th root of unity.  Note that $S\in \CM R$ since $\dim R=2$ so reflexive equals CM, thus $S$ is a Gorenstein ring by Proposition \ref{gorensteiness}\eqref{gorensteiness 1}. Then $\Z_n=\langle\zeta\rangle$ acts on $S=\bigoplus_{i\in\Z_n}S_i$ by $\zeta\cdot(x_i)_{i\in\Z_n}:=(\zeta^ix_i)_{i\in\Z_n}$, and the invariant subring is $R$. If $S$ is a rational singularity, then it is a quotient singularity and thus so is $R$, a contradiction. Thus $S$ is Gorenstein but not a quotient singularity, and so $S$ does not have an NCCR, which is precisely an NCR since $S$ is Gorenstein and $\dim S=2$.
The last statement follows from \cite[Corollary 3.3]{DITV}.
\end{proof}

Theorem~\ref{intro equiv} from the introduction is a special case ($N\colonequals S$) of the following result,
where for an $R$-algebra $A$, we denote by $\CM A$ the category of $A$-modules which are Cohen-Macaulay as an $R$-module.

\begin{theorem}\label{main equivalence text}
Let $R$ be a CM local normal domain with a canonical module $\omega_R$ which is $\Q$-Gorenstein of index $n$. Let $S$ be a canonical cover of $R$.
If $N\in\refl^{\Zn}\!S$ gives a GM of $S$, then there are equivalences
\[
\CM^{\Zn}\!\End_S(N)\simeq\CM\End_R(N)\simeq\ZZ(R,N)\colonequals \{X\in\refl R\mid\Hom_R(N,X)\in\CM R\}.
\]
\end{theorem}
\begin{proof}
The first equivalence follows from Proposition \ref{E and Sn}\eqref{E and Sn 2}, and the second is the obvious restriction of the equivalence $\Hom_R(N,-)\colon\refl R\to\refl\End_R(N)$ in Proposition \ref{ref eq}.
\end{proof}

\begin{example}
Let $R=k[x,y,z]^{(6)}$ be the 6th Veronese subring of $k[x,y,z]$.
Then $S=k[x,y,z]^{(3)}$ and there is an equivalence
\[
\CM^{\Z_2}\!S\simeq\{X\in\CM R\mid(\omega_R\otimes_RX)^{**}\in\CM R\}.
\]
Both $R$ and $S$ have infinite CM type.
\end{example}


\subsection{Applications}\label{app section}
We now give a proof of the three applications in the introduction. 
Our first application is the following, which complements \cite{SpV}.

\begin{corollary}\label{SVdB question}
Let $G$ be a connected reductive group in characteristic zero, $W$ a $G$-representation. If $G$ acts generically on $X\colonequals \Spec SW$ such that $R\colonequals k[X]^G$ is not Gorenstein, then $R$ is not $\Q$-Gorenstein. Therefore $R$ admits neither a GM nor an NCCR.
\end{corollary}

\begin{proof}
The canonical module $\omega_R$ is isomorphic to $(SW\otimes_k\upchi)^G$ for some character $\upchi$ \cite[Korollar 2]{Knop}.  Since the action is generic, by \cite[Corollary 4.1.4]{SpV} the natural functor
\begin{equation}
\mod(G,k)\to \refl R\label{nat funct refl}
\end{equation}
is symmetric monoidal, and so $\omega_R^n\cong (SW\otimes_k\upchi^n)^G$. Since  the action contains a fixed point, by \cite[Lemma 4.1.5]{SpV} the functor \eqref{nat funct refl} reflects isomorphisms.

On the other hand, since $R$ is a $\Z$-graded ring such that $\dim_kR_i$ is finite for any $i\in\Z$, the category $\mod^{\Z}\!R$ is Krull-Schmidt. Therefore, if $R$ is $\mathds{Q}$-Gorenstein, then $\omega_R^n$ is projective for some $n>0$, and hence $\omega_R^n\cong R$ holds.  Thus in this case $(SW\otimes_k\upchi^n)^G\cong (SW)^G$ holds, so we deduce that $\upchi^n$ is the trivial representation.

Since $G$ is a connected algebraic group, the character group of $G$ is torsion-free by  \cite[Lemma B, p104]{H}. Thus $\upchi$ is the trivial representation and we have $\omega_R\cong R$, a contradiction.
Consequently, $R$ cannot be $\mathds{Q}$-Gorenstein, and admits neither a GM nor an NCCR by Theorem \ref{prop 1}\eqref{prop 1 2}.
\end{proof}

\begin{remark}
It follows that quotient singularities need not have any NCCRs. A concrete example is the Francia flip, namely the group $\mathbb{C}^*$ acting on the polynomial ring $S=\mathbb{C}[x_1,x_2,y_1,y_2]$ with weights $(2,1,-1,-1)$, and $R\colonequals S^{\mathbb{C}^*}$. There are many other examples, including determinantal varieties from non-square matrices \cite{BLV}.
\end{remark}

\begin{remark} 
The above also shows that the existence of an NCR does not imply that $R$ is $\Q$-Gorenstein, so that Theorem \ref{prop 1}\eqref{prop 1 2} cannot be generalized.  Indeed, the Francia flip $R$ is CM-finite, and hence the additive generator $M$ of $\CM R$ gives an NCR. Another example of an NCR of $R$ is given by the resolution $X\to\Spec R$, which has a tilting bundle $\mathcal{V}$ \cite{VdB1d} with $\End_X(\mathcal{V})\cong\End_{S_0}(S_0\oplus S_{-1})$. Since $X$ is smooth, this ring has finite global dimension.
\end{remark}

\begin{remark}
We also remark that the existence of a maximal modification algebra (MMA) (see \cite{IW1}) does not imply that $R$ is $\mathds{Q}$-Gorenstein. In fact, the Francia flip $R$ has an MMA since it is CM-finite.
\end{remark}

Before proving our second application, let us first recall the Buchweitz--Greuel--Schreyer classification based on Kn\"{o}rrer's work.

\begin{theorem}\cite{BGS}\label{bgs}
Let $R$ be a complete local Gorenstein ring with algebraically closed residue field of characteristic zero.
Then $R$ is CM-finite if and only if it is a simple singularity.
\end{theorem}

 The proof of Corollary \ref{Auslander-Esnault} is based on our main Theorem \ref{main NCCR}.

\begin{proof}[Proof of Corollary \ref{Auslander-Esnault}]
Let $M$ be an additive generator of $\CM R$. Then $M$ gives an NCCR of $R$, and therefore $\widehat{M}_{\pp}$ gives an NCCR of $\widehat{R}_{\pp}$ for any $\pp\in\Spec R$. In the rest, we denote $\widehat{R}_{\pp}$ for $\pp\in\Spec R$ with $\dim \widehat{R}_{\pp}=2$ by $R$, and prove that it is a quotient singularity.

Since $R$ is complete local, it admits a canonical module \cite[Corollary 3.3.8]{BH}.
By Theorems~\ref{prop 1}\eqref{prop 1 2} and \ref{main}\eqref{main 2}, it follows that $R$ is $\mathds{Q}$-Gorenstein (and thus $\mathds{Q}$-CY since $R$ is local), and its canonical cover $S$ is Gorenstein and has an NCCR. Therefore $S$ is CM-finite, and further complete local by Proposition \ref{construct S}\eqref{construct S 4}.
By Theorem \ref{bgs}, $S$ must be a simple singularity, and hence a quotient singularity. Since $R$ is the invariant subring of an action of $\Zn$ on $S$, the result follows.
\end{proof}

We finally prove our third application, Corollary \ref{log-terminal}.

\begin{proof}[Proof of Corollary \ref{log-terminal}]
Since $R$ is affine, it admits a canonical module $\omega_R$.  Since $R$ admits an NCCR, by Theorem \ref{prop 1}\eqref{prop 1 2},  $R$ is $\mathds{Q}$-Gorenstein.  Thus there is a finite affine open cover $\Spec R=\bigcup_i\Spec R_i$ such that each $R_i$ is $\mathds{Q}$-CY.  Let $S_i$ be a canonical cover of $R_i$.

Since NCCRs localize, $R_i$ has an NCCR.  By Theorem~\ref{main}, it follows that $S_i$ has a NCCR $\Lambda_i$.  But since $S_i$ is normal by Proposition \ref{gorensteiness}\eqref{gorensteiness 2}, the centre of $\Lambda_i$ is precisely $S_i$.  Since $\Lambda_i$ is thus a finitely generated homologically homogeneous $k$-algebra, $S_i$ has at worst rational singularities by \cite[Theorem 1.1]{SV}.  It follows that $R_i$ is log-terminal \cite[Proposition 1.7]{K}, and since this holds for all $i$,  $R$ is log-terminal.
\end{proof}

\end{document}